\documentclass{article}

\usepackage[utf8]{inputenc}		

\usepackage{amsmath}
\usepackage{amsfonts}
\usepackage{amssymb}
\usepackage{amsthm}

\usepackage{graphicx}
\usepackage{xcolor}

\usepackage{epsfig}
\usepackage{amsbsy}
\usepackage{textcomp}
\usepackage{comment}
\usepackage{hyperref}
\usepackage[all]{hypcap}
\usepackage{aliascnt}

\theoremstyle{plain}
\newtheorem{thm}{Theorem}[section]
\newaliascnt{lem}{thm}
\newtheorem{lem}[lem]{Lemma}
\aliascntresetthe{lem}
\newaliascnt{pro}{thm}

\aliascntresetthe{pro}
\newaliascnt{cor}{thm}
\newtheorem{cor}[cor]{Corollary}
\aliascntresetthe{cor}
\newaliascnt{que}{thm}

\aliascntresetthe{que}
\newaliascnt{rem}{thm}

\aliascntresetthe{rem}
\newaliascnt{con}{thm}
\newtheorem{con}[con]{Conjecture}
\aliascntresetthe{con}
\newcounter{clmC}[thm]
\newtheorem{clm}[clmC]{Claim}
\newtheorem*{clmnn}{Claim} 
\newtheorem*{LinkageThm}{Linkage Theorem}
\newtheorem*{thm1}{\autoref{thm:dagger2matroid}}
\newtheorem*{thm2}{\autoref{thm:noAC2Gammoid}}
\newtheorem{mainthm}{Theorem}

\theoremstyle{definition}
\newaliascnt{defi}{thm}
\newtheorem{defi}[defi]{Definition}
\aliascntresetthe{defi}
\newaliascnt{exm}{thm}
\newtheorem{exm}[exm]{Example}
\aliascntresetthe{exm}

\newcommand{\mcm}[3]{\newcommand{#1}[#2]{{\ensuremath{#3}}}} 

\mcm{\tuple}{1}{\langle #1 \rangle}
\mcm{\name}{1}{\ulcorner #1 \urcorner}
\mcm{\Fbb}{0}{\mathbb{F}}
\mcm{\Nbb}{0}{\mathbb{N}}
\mcm{\Zbb}{0}{\mathbb{Z}}
\mcm{\Rbb}{0}{\mathbb{R}}
\mcm{\Cbb}{0}{\mathbb{C}}
\mcm{\Acal}{0}{\cal A}
\mcm{\Bcal}{0}{\cal B}
\mcm{\Ccal}{0}{\cal C}
\mcm{\Dcal}{0}{\cal D}
\mcm{\Ecal}{0}{\cal E}
\mcm{\Fcal}{0}{\cal F}
\mcm{\Gcal}{0}{\cal G}
\mcm{\Hcal}{0}{\cal H}
\mcm{\Ical}{0}{\cal I}
\mcm{\Lcal}{0}{\cal L}
\mcm{\Mcal}{0}{\cal M}
\mcm{\Wcal}{0}{\cal W}
\mcm{\Ncal}{0}{\cal N}
\mcm{\Pcal}{0}{{\cal P}}
\mcm{\Qcal}{0}{{\cal Q}}
\mcm{\Scal}{0}{{\cal S}}
\mcm{\Tcal}{0}{{\cal T}}
\mcm{\Ucal}{0}{{\cal U}}
\mcm{\Vcal}{0}{{\cal V}}
\mcm{\Mfrak}{0}{\mathfrak M}

\mcm{\restric}{0}{\upharpoonright}
\mcm{\upset}{0}{\uparrow}
\mcm{\onto}{0}{\twoheadrightarrow}
\mcm{\smallNbb}{0}{{\small \mathbb{N}}}
\DeclareMathOperator{\preop}{op}
\mcm{\op}{0}{^{\preop}}

{\begin{array}{c}
\setlength{\unitlength}{1em}}%
{\end{array}}

\renewcommand{\-}{\setminus}
\newcommand{\Ter}{\mbox{Ter}}
\newcommand{\Ini}{\mbox{Ini}}
\newcommand{\ini}{\mbox{Ini}}

\begin{document}
\author{Hadi Afzali \and Hiu-Fai Law \and Malte M\"uller\thanks{Fachbereich Mathematik, Universit\"at Hamburg}}

\title{Infinite strict gammoids}
\date{\today}
\maketitle
\setcounter{section}{-1}

\begin{abstract}
Finite strict gammoids, introduced in the early 1970's, are matroids defined via finite digraphs equipped with some set of sinks:~a set of vertices is independent if it admits a linkage to these sinks. An independent set is maximal precisely if it admits a linkage onto the sinks. 

In the infinite setting, this characterization of the maximal independent sets need not hold. 
We identify a type of substructure as the unique obstruction to the characterization. We then show that the sets linkable onto the sinks form the bases of a (possibly non-finitary) matroid precisely when the substructure does not occur. 
\end{abstract}

\section{Introduction}
Infinite matroid theory has seen vigorous development since Bruhn et al \cite{BDKPW} in 2010 gave five equivalent sets of axioms for infinite matroids in response to a problem proposed by Rado \cite{Rad66} (see also Higgs \cite{Hig69} and Oxley \cite{Oxl78}). 
In this paper, we continue this ongoing project by focusing on the strict gammoids, which originated from the transversal matroids introduced by Edmonds and Fulkerson \cite{EF65}. A transversal matroid can be defined by taking as its independent sets the subsets of a fixed vertex class of a bipartite graph matchable to the other vertex class. 
Perfect \cite{Per68} generalized transversal matroids to gammoids by replacing matchings in bipartite graphs with disjoint directed paths in digraphs. 
Later, Mason \cite{Mas72} started the investigation of a subclass of gammoids known as {\it strict\/}  gammoids. 

To be precise, 
let a \emph{dimaze} (short for \emph{di}rected \emph{maze}) be a digraph with a fixed subset of the vertices of out-degree 0, called \emph{exits}. A set of vertices of (the digraph of) the dimaze is called \emph{independent} if it is \emph{linkable} to the exits by a \emph{linkage}, i.e.~a collection of disjoint directed paths. The set of all linkable sets is the \emph{linkability system} of the dimaze. 
Any matroid defined by (the linkability system on the vertex set of) a dimaze is a \emph{strict gammoid}. 

Mason \cite{Mas72} proved that every finite dimaze defines a matroid. 
When a dimaze is infinite, Perfect \cite{Per68} gave sufficient conditions for when  some subset of the linkability system gives rise to a matroid. Any such matroid is \emph{finitary}, in the sense that a set is independent as soon as all its finite subsets are. 
Since finitary matroids were the only ones known at that time, infinite dimazes whose linkability systems are non-finitary\footnote{For example, in the dimaze on an infinite star directed from the centre towards the exits at the leaves, every finite set of vertices is independent, but the whole vertex set is not.} were not considered to define matroids. 

With infinite matroids canonically axiomatized in a way that allows for non-finitary matroids, a natural question is whether every infinite dimaze now defines a matroid. In general, the answer to this question is still negative, as the linkability system may fail to satisfy one of the infinite matroid axioms (IM), which asks for the existence of certain maximal independent sets. 
Observe that in any finite dimaze, a set is linkable onto the exits if and only if it is maximally independent.\footnote{Note that an independent set that is not linkable onto the exits admits a linkage which misses some of the exits. Adding these exits via trivial paths shows that the independent set is not maximal.} 
However, in an infinite dimaze, sets which are linkable onto the exits need not be maximally independent. But if they all are, the dimaze defines a matroid:

\begin{thm1}
Given a dimaze, suppose that every set that is linkable onto the exits is maximally independent. Then the linkability system is the set of independent sets of a matroid. 
\end{thm1} 

So the question arises:~in which dimazes is every set that is linkable onto the exits maximally independent? 
Investigation of this question leads us to the following example of a dimaze. 
An \emph{alternating ray} is a digraph having infinitely many vertices of in-degree 2 such that the underlying graph is a \emph{ray}, i.e.~an 1-way infinite path. 
An \emph{alternating comb} is a dimaze constructed by linking all the vertices of in-degree 2 and the first vertex of an alternating ray onto a set of exits by (possibly trivial) disjoint directed paths which meet the ray exactly at their initial vertices (see for example \autoref{fig:ARIC}a). 
It will be easy to see that, in an alternating comb, the set of vertices of out-degree 2 can be linked to the exits in two different ways, either onto or to a proper subset; hence, the set is not maximally independent and yet is linkable onto the exits. 

By proving that alternating combs form the unique obstruction to the  characterization of maximal independent sets as sets linkable onto the exits, we are able to prove the following theorem. When we say that a dimaze \emph{contains} an alternating comb, we require that the set of exits of the former includes that of the latter. 

\begin{thm2}
Given a dimaze, the vertex sets linkable onto the exits form the bases of a matroid if and only if the dimaze contains no alternating comb. The independent sets of this matroid are precisely the linkable sets of vertices.
\end{thm2}

The non-trivial direction of \autoref{thm:noAC2Gammoid} implies that a dimaze whose linkability system fails to define a matroid contains an alternating comb. 
Conversely, a dimaze containing an alternating comb may still define a matroid. However, the set of bases of such a matroid is a proper subset of the sets linkable onto the exits, and can be difficult to describe.  

\bigskip
We collect definitions and give examples of infinite dimazes which do not define a matroid in \autoref{sec:preliminaries}.  
In \autoref{sec:strictGammoids}, we first prove that the independence augmentation axiom holds in general. 
Then after rephrasing a proof of the linkage theorem of Pym \cite{Pym69}, we 
prove Theorems 1 and 2.  
In \autoref{sec:withAR}, we construct a strict gammoid which cannot be defined by any dimaze without alternating combs. 
There are two intermediate steps, one leads to a connectivity result; the other shows that every infinite tree, viewed as a bipartite graph, gives rise to a transversal matroid, a statement which is of independent interest.

\section{Preliminaries} \label{sec:preliminaries}
In this section, we present relevant  definitions. For notions not found here, we refer to \cite{BDKPW} and \cite{Oxl92} for  matroid theory, and \cite{Die10} for graph theory.

Given a set $E$ and a family of subsets $\Ical\subseteq 2^E$, let $\Ical^{\max}$ denote the maximal elements of $\Ical$ with respect to set inclusion. For a set $I\subseteq E$ and $x\in E$, we also write $I+x, I-x$ for $I\cup \{x\}$ and $I\-\{x\}$ respectively. 
\begin{defi}\cite{BDKPW}
A \emph{matroid} $M$ is a pair $(E, \Ical)$ where $E$ is a set and $\Ical\subseteq 2^E$ which satisfies the following:
\begin{samepage}
\begin{enumerate}
\item[(I1)] $\emptyset\in \Ical$.
\item[(I2)] If $I\subseteq I'$ and $I'\in\Ical$, then $I\in \Ical$.
\item[(I3)] For all $I\in \Ical \setminus\Ical^{\max}$ and $I'\in \Ical^{\max}$, there is an $x\in I'\setminus I$ such that $I+x\in \Ical$.
\item[(IM)] Whenever $I\in \Ical$ and $I\subseteq X\subseteq E$, the set $\{I'\in \Ical : I\subseteq I'\subseteq X\}$ has a maximal element. 
\end{enumerate}
\end{samepage}
\end{defi}

The set $E$ is the \emph{ground set} and the elements in $\Ical$ are the {\it independent sets} of $M$. 
Equivalently, matroids can be defined with base axioms. A collection of subsets $\Bcal$ of $E$ is the set of bases of a matroid if and only if the following three axioms hold:
\begin{enumerate}
\item[(B1)] $\Bcal\neq \emptyset$.
\item[(B2)] Whenever $B_1, B_2\in\Bcal$ and $x\in B_1\- B_2$, there is an element $y$ of $B_2\- B_1$ such that $(B_1-x)+y\in \Bcal$.
\item[(BM)] The set $\Ical$ of all subsets of elements in $\Bcal$ satisfies (IM).
\end{enumerate}

The ground sets we consider will be sets of vertices of digraphs or  bipartite graphs. Unless there is danger of confusion of the ground set, we identify a matroid with its set of independent sets.

\medskip
\begin{defi}
\label{def:dimaze}
Let $D$ be a digraph and $B_0$ a subset of the vertices the of out-degree $0$.\footnote{The assumption on $B_0$ incurs no loss of generality, as we may delete the out-going edges from $B_0$ without changing the linkability system. Moreover, this assumption excludes unwanted trivial cases in subsequent constructions by forcing vertices having out-going edges to lie outside $B_0$. 
} The pair $(D, B_0)$ is called a \emph{dimaze} and $B_0$ is the set of \emph{exits}. 
\end{defi}

Given a dimaze $(D, B_0)$, a \emph{linkage} $\Pcal$ is a set of disjoint directed paths such that the terminal vertex of each path is in $B_0$. Let $\Ini(\Pcal)$ and $\Ter(\Pcal)$ be respectively the set of initial vertices and that of terminal vertices of paths in $\Pcal$. A set $I\subseteq V(D)$ is \emph{linkable} or \emph{independent} if there is a linkage $\Pcal$ \emph{from} $I$, i.e.~$\Ini(\Pcal)=I$. Suppose further that $\Ter(\Pcal)=B_0$, then $I$ is linkable \emph{onto} $B_0$. 
The collection of linkable sets is called the \emph{linkability system}. 
Note that:
\begin{align}
\label{eqn:linkableOnto}
\mbox{Any linkable set in $(D, B_0)$ can be extended to one linkable onto $B_0$.} 
\end{align}

\begin{defi}
\label{def:linkabilitySystem}
Given a dimaze $(D, B_0)$, we denote the pair $V:=V(D)$ and the linkability system by $M_L(D, B_0)$. (The linkability system of) a dimaze \emph{defines} a matroid if $M_L(D, B_0)$ is a matroid.\footnote{ 
In particular, $B_0$ is always a base when $M_L(D, B_0)$ is a matroid.} Any matroid arising in this way is called a {\it strict gammoid}. 
\end{defi}

Different dimazes may define the same strict gammoid. For example, a rank one uniform matroid can be defined via a star or a path. See \autoref{fig:ARIC} for another pair of examples.

Recall that given $A, B\subseteq V$, an \emph{A--B separator} $S$ is a set of vertices such that there are no paths from $A$ to $B$ avoiding $S$. A separator is \emph{on} a linkage $\Pcal$ if it consists of exactly one vertex on each path in $\Pcal$. The celebrated Aharoni-Berger-Menger's theorem \cite{AB09} states that there exist a linkage from a subset of $A$ to $B$ and an $A$--$B$ separator on this linkage.

 Mason \cite{Mas72} (see also \cite{Per68}) showed that given a finite digraph $D$, for any $B_0\subseteq V$, $M_L(D, B_0)$ is a matroid. However, this is not the case for infinite digraphs. 
For example, let $D$ be a complete bipartite graph between an uncountable set $X$ and a countably infinite set $B_0$ with all the edges directed towards $B_0$. Then $I\subseteq X$ is independent if and only if $I$ is countable, so there is not any maximal independent set in $X$. Hence, $M_L(D, B_0)$ does not satisfy the axiom (IM). 

\begin{exm}
\label{exm:halfGrid}
A counterexample with a locally finite digraph is the half-grid. Define a digraph $D$ by directing upwards or leftwards the edges of the subgraph of the grid $\mathbb Z_{\geq 0} \times \mathbb Z_{\geq 0}$ 
induced by $\{(x, y): y>0 \mbox{ and }y\geq x\geq 0\}$. The \emph{half-grid} is the dimaze $(D, B_0)$ where $B_0:=\{(0, y): y  > 0\}$; see \autoref{fig:halfGrid}. Then $I:=\{(x, x): x>0\}$ is linkable onto a set $J\subseteq B_0$ if and only if $J$ is infinite. Therefore, $I\cup (B_0\- J)$ is independent if and only if $J$ is infinite. Hence, $I$ does not extend to a maximal independent set in $X:=I\cup B_0$. 
\end{exm}
\begin{figure}
\centering
 \includegraphics{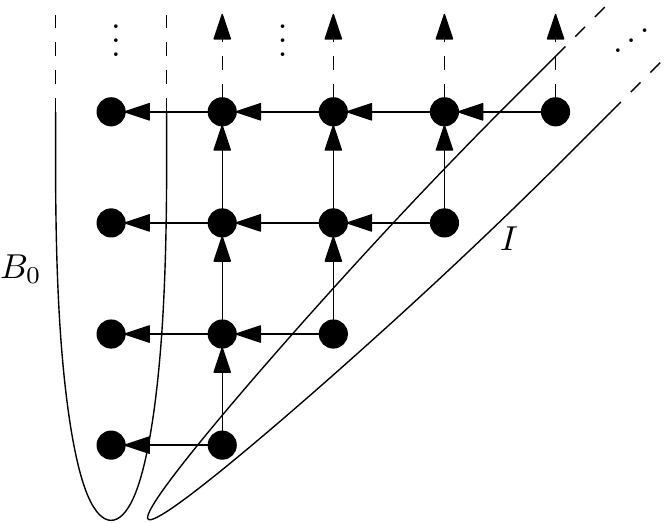}
 \caption{A locally finite dimaze which does not define a matroid}
\label{fig:halfGrid} 
\end{figure}

We now define families of digraphs and dimazes central to our investigations. 
An \emph{alternating ray} ($AR$) is a digraph having infinitely many vertices of in-degree 2 such that the underlying graph is a \emph{ray}, i.e.~an 1-way infinite path. An \emph{alternating comb} ($AC$) is a dimaze constructed by linking all the vertices of in-degree 2 and the first vertex of an alternating ray onto a set of exits by (possibly trivial) disjoint (directed) paths which meet the ray exactly at their initial vertices (see for example \autoref{fig:ARIC}a). 
We say that a dimaze $(D, B_0)$ \emph{contains} an alternating comb, if $D$ contains the digraph of an AC as a subdigraph and $B_0$ contains the exits of the AC.  

Given a path $P$ and a vertex $w$ on $P$, $Pw$ denotes the segment from the initial vertex up to $w$ and $P\!\!\stackrel{\circ}{w}$ the same segment with $w$ excluded. We use $PwQ$ for the concatenation of $Pw$ and $wQ$ where $Q$ is a path containing $w$; and other similar notations. We also identify $P$ with its vertex set.

\section{Dimazes and matroid axioms}\label{sec:strictGammoids}
The aim of this section is to give a sufficient condition for a dimaze $(D, B_0)$ to define a matroid. As (I1) and (I2) hold for $M_L(D, B_0)$, we need only consider (I3) and (IM). 

\subsection{Linkability system and proof of (I3)}
We prove that (I3) holds in any $M_L(D, B_0)$ using a result due to Gr\"unwald \cite{Gru38}.\footnote{A presentation of the lemma can be found in \cite[Lemmas 3.3.2 and 3.3.3]{Die10}. With alternating walks defined as here, the proofs given there work almost verbatim for general dimazes.}
Let $(D,B_0)$ be a dimaze and $X$ a fixed subset of $V$. Fix a linkage $\Pcal$ from a subset of $X$ to $B_0$. An \emph{alternating walk} $W=v_0e_0v_1e_1 \ldots e_{n-1}v_n$ (\emph{with respect to} $\Pcal$) is a sequence alternating between vertices and edges of $D$ such that for all $i,j< n$ with $i\neq j$, every edge $e_i$ is incident with the vertices $v_i$ and $v_{i+1}$, $e_i\neq e_j$, $v_0\in X\-V(\Pcal)$, and  the following properties hold:
\begin{enumerate}
\item $e_i=(v_{i+1}, v_i)$ if and only if $e_i\in E(\Pcal)$;
\item if $v_i=v_j$ then $v_i\in V(\Pcal)$;
\item if $v_i\in V(\Pcal)$, then $\{e_{i-1},e_i\} \cap E(\Pcal)\neq \emptyset$ (with $e_{-1}:=e_0$).
\end{enumerate}

\begin{lem} \label{thm:AW2Linkage}
Let $(D,B_0)$ be a dimaze, $X$ a subset of $V$ and $\Pcal$ a linkage from a subset of $X$ to $B_0$. If an alternating walk ends in $B_0\-V(\Pcal)$, then $D$ contains a linkage $\Qcal$ such that $\Ini(\Pcal)\subsetneq \Ini(\Qcal)\subseteq X$ and $\Ter(\Pcal)\subsetneq \Ter(\Qcal)\subseteq B_0$. If no such walk exists, then there is an $X$--$B_0$ separator on $\Pcal$.
\end{lem}

\begin{lem}\label{thm:I3_ForLinkabilitySystems}
Let $(D,B_0)$ be a dimaze. Then $M_L(D,B_0)$ satisfies (I3).
\end{lem}
\begin{proof}
Let $I, B\in M_L(D, B_0)$ such that $B$ is maximal but $I$ is not. Then we have a linkage $\Qcal$ from $B$ and another $\Pcal$ from $I$. We may assume $\Pcal$ misses some $v_0\in B_0$. 

If there is an alternating walk with respect to $\Pcal$ from $(B\cup I)\setminus V(\Pcal)$ to $B_0\-V(\Pcal)$, then by \autoref{thm:AW2Linkage}, we can extend $I$ in $B\-I$.

On the other hand, if no such walk exists, we draw a contradiction to the maximality of $B$. In this case, by \autoref{thm:AW2Linkage}, there is an $(B\cup I)$--$B_0$ separator $S$ on $\Pcal$. For every $v\in B$, let $Q_v$ be the path in $\Qcal$ starting from $v$. Let $s_v$ be the first vertex of $S$ that $Q_v$ meets and $P_v$ the path in $\Pcal$ containing $s_v$. Let us prove that the set $\Qcal'=\{Q_vs_vP_v:v\in B\}$ is a linkage.

Suppose $v$ and $v'$ are distinct vertices in $B$ such that  $Q_vs_vP_v$ and $Q_{v'}s_{v'}P_{v'}$ meet each other.
As $\Pcal$ and $\Qcal$ are linkages, without loss of generality, we may assume $Q_vs_v$ meets $s_{v'} P_{v'}$ at some $s\notin S$. 
Then $Q_vsP_{v'}$ is a path from $B$ to $B_0$ avoiding the separator. This contradiction shows that $\Qcal'$ is indeed a linkage from $B$ to $B_0$. As $\Qcal'$ does  not cover $v_0$, $B+v_0$ is independent which contradicts to the maximality of $B$.
\end{proof}

\subsection{Linkage theorem and (IM)}
\label{sec:linkageThm}
Since (I3) holds for any $M_L(D, B_0)$, it remains to investigate (IM). 
Recall that for any finite digraph $D$ and $B_0\subseteq V$, the following holds for $(D, B_0)$: 
\begin{align}
\label{eqn:dagger} 
\mbox{A set is maximally independent if and only if it is linkable onto the exits.} \tag{$\dagger$}
\end{align}
When $D$ is infinite, \eqref{eqn:dagger} need not hold;~for instance, the half grid in \autoref{exm:halfGrid}, which does not even define a matroid. It turns out that there are always non-maximal independent sets linkable onto the exits in a dimaze that does not define a matroid. To prove this, we will use the linkage theorem \cite{Pym69} (see also \cite{DT06}) and the infinite Menger's theorem \cite{AB09}. 

Now the natural question is:~in which dimazes is every set, that is linkable onto the exits, a maximal independent set? Consider the AC given in  \autoref{fig:ARIC}a. Using the notation there, the set $X=\{x_i:i\geq 1\}$ can be linked onto $B_0$ by the linkage $\{(x_i, y_{i-1}): i\geq 1\}$ or to $B_0-x_0$ by the linkage $\{(x_i, y_i): i\geq 1\}$. Hence, $X$ is a non-maximal independent set that is linkable onto $B_0$.  More generally, if a dimaze $(D, B_0)$ contains an alternating comb $C$, then the vertices of out-degree 2 on $C$ together with $B_0-C$ is a non-maximal set linkable onto $B_0$. So an answer to the above question must exclude dimazes containing an AC. 
We will prove that dimazes without any AC are precisely the answer. 

One might think that the following proof strategy should work:~If the characterization of maximal independent sets does not hold, then there are two linkages, a blue one from a set and a red one from a proper superset, both covering the exits. To construct an AC, one starts with finding an AR. For that, 
a first attempt is to ``alternate'' between the red and blue linkages, i.e.~to repeat the following: go forward along the red linkage, change to the blue one at some common vertex, and then go backwards on the blue linkage, and change again to the red one. It is not the case that this construction always gives rise to an AR (because vertices might be visited twice). But supposing that we do get an AR, a natural way to extend it to an AC is to use the terminal segments of one fixed linkage. However, this AR can have two distinct vertices of in-degree 2 which lie on the same path of the fixed linkage. 

Appropriate choices to alternate between the linkages will be provided by the \emph{proof} of the linkage theorem of Pym \cite{Pym69}. 
So we give a sketch of the proof, rephrased for our purpose. 

\begin{LinkageThm}\phantomsection\label{thm:linkageThm}
Let $D$ be a digraph and two linkages be given: the ``red'' one, $\Pcal = \{ P_x : x \in X_\Pcal \}$, from $X_\Pcal$ onto $Y_\Pcal$ and the ``blue'' one, $\Qcal = \{ Q_y : y \in Y_\Qcal \}$, from $X_\Qcal$ onto $Y_\Qcal$. Then there is a set $X^\infty$ satisfying $X_\Pcal \subseteq X^\infty \subseteq X_\Pcal \cup X_\Qcal$ which is linkable onto a set $Y^\infty$ satisfying $Y_\Qcal \subseteq Y^\infty \subseteq Y_\Qcal \cup Y_\Pcal$.
\end{LinkageThm}
\begin{proof}[Sketch of proof]\footnote{
We find that the following story makes the proof more intuitive. Imagine that   
a directed path in a linkage corresponds to a pipeline which transports water backwards from a pumping station located at the terminal vertex of this path to its initial vertex. At the initial vertex of every red pipeline, there is a farm whose farmer is \emph{happy} if and only if his farm is supplied with water. 
The water from every blue pipeline flows into the desert at the initial vertex, even if there is a farm. 

The story starts on day $0$ when suddenly all the red pumping stations are broken. At the beginning of each day, every farmer follows the rule: 

\begin{quote}
``If you are \emph{unhappy}, then move onward along your red pipeline until you can potentially get some water.'' 
\end{quote}
So the unhappy farmers take their toolboxes and move along their pipelines. 
Every farmer stops as soon as he comes across a pipeline/pumping station which still transports water and manipulates it in such a way that all the water flows into his red pipeline and then to his farm. If a pipeline has been manipulated by more than one farmer, then the one who is the closest to the pumping station gets the water, ``stealing'' it from the others. 
When a farmer arrives at his pumping station, he repairs it, but only if he cannot steal. We assume that every farmer needs a full day for the whole attempt to get water; so he 
realizes that someone else has stolen his water only at the end of the day. The story ends when every farmer is happy. 

A proof of the linkage theorem can be derived by examining the flow of water and the movement of the farmers. 
In particular, each pumping station supplies its final farm and each farmer reaches his final position (where he stays happily) after only finitely many days. 
It turns out that the water flow at the end is a linkage covering the red initial and the blue terminal vertices, as required by the theorem. 
} 
We construct a sequence of linkages converging to a linkage with the desired properties. For each integer $i\geq 0$, we will specify a vertex on each path in $\Pcal$. For each $x\in X_\Pcal$, let $f^0_x := x$. Let $\Qcal^0 := \Qcal$.
For each $i> 0$ and each $x\in X_\Pcal$, let $f^i_x$ be the last vertex $v$ on $f^{i-1}_xP_x$ such that $(f^{i-1}_x P_x\!\stackrel{\circ}{v}) \cap V(\Qcal^{i-1}) = \emptyset$.\footnote{The farmer at $f^{i-1}_x$ moves onward to the closest vertex where there was still water at the beginning of day $i-1$ or, if no such vertex exists, to the pumping station of $P_x$.} 
For $y\in Y_\Qcal$, let $t_y^i$ be the first vertex $v\in Q_y$ such that the terminal segment $\stackrel{\circ}{v}\!Q_y$ does not contain any $f^i_x$.\footnote{If $f^i_x=t_y^i$, the farmer at $f^i_x$ is the closest to the pumping station at $y$ and steals the water from any other farmer on $Q_y$.} Let

\vspace{-0.2in}
\begin{align*}
\Acal^{i} &:= \{ Q_y\in \Qcal : t_y^i \neq f_x^i\ \forall x\in X_\Pcal\},\\
\Bcal^{i} &:= \{ P_xf^i_xQ_y: x\in X_\Pcal, y\in Y_\Qcal \mbox{ and } f^i_x =t_y^i\},\\
\Ccal^{i} &:= \{ P_x\in \Pcal: f_x^i\in Y_\Pcal \mbox{ and }f_x^i\neq t_y^i\ \forall y\in Y_\Qcal\},
\end{align*}
and $\Qcal^i := \Acal^i \cup \Bcal^i \cup \Ccal^i$. 

Inductively, one can show that $\Qcal^i$ is a linkage which covers $Y_\Qcal$.  
Since for any $x$ and $j\geq i$, $P_xf^i_x\supseteq P_xf_x^j$, it can be shown that $t_y^iQ_y\subseteq t_y^jQ_y$ for any $y$. As all the paths $P_x$ and $Q_y$ are finite, there exist integers $i_x, i_y\geq 0$ such that $f^{i_x}_x = f^k_x$, $t^{i_y}_y = t^l_y$ for all integers $k\geq i_x$ and $l\geq i_y$. Thus, we define $f^\infty_x := f^{i_x}_x, t^\infty_y := t^{i_y}_y$ and 

\vspace{-0.2in}
\begin{align*}
\Acal^{\infty} &:= \{ Q_y\in \Qcal : t_y^\infty \neq f_x^\infty\ \forall x\in X_\Pcal\},\\
\Bcal^{\infty} &:= \{ P_xf^\infty_xQ_y: x\in X_\Pcal, y\in Y_\Qcal \mbox{ and } f^\infty_x =t_y^\infty\},\\
\Ccal^{\infty} &:= \{ P_x\in \Pcal: f_x^\infty\in Y_\Pcal \mbox{ and }f_x^\infty\neq t_y^\infty\ \forall y\in Y_\Qcal\}.
\end{align*}
Then $\Qcal^\infty:=\Acal^\infty \cup \Bcal^\infty \cup \Ccal^\infty$ covers $Y_\Qcal$. Moreover, $\Qcal^\infty$ is a linkage. Indeed, as $t^\infty_y Q_y \subseteq t^i_yQ_y$ for any $i$, the intersection of $P_x f^\infty_x$ and  $t^\infty_y Q_y$ is either empty or the singleton of $f^\infty_x= t^\infty_y$. It remains to argue that $X_{\Pcal}\subseteq \Ini(\Qcal^\infty)$. Let $x\in X_{\Pcal}$. 
If $f^\infty_x=t^\infty_y$ for some $y$, then $x\in \Ini(\Bcal^\infty)$. Otherwise, there exists an integer $j$ such that $f^\infty_x = f^j_x$ and $f^\infty_x \neq t^j_y$ for any $y$. Since $f^{j+1}_x = f^j_x$, it follows that $f^j_x$ is on a path in $\Ccal^j$, so $f^\infty_x \in Y_\Pcal$. Hence, $x\in \Ini(\Ccal^\infty)$. 
\end{proof}

We can now prove the following. 
\begin{mainthm}
\label{thm:dagger2matroid}
Given a dimaze $(D, B_0)$, suppose that every independent set linkable onto the exits is maximal, then the dimaze defines a matroid. 
\end{mainthm}

\begin{proof}
Since (I1) and (I2) are obviously true for $M_L(D, B_0)$, and that (I3) holds by \autoref{thm:I3_ForLinkabilitySystems}, to prove the theorem, it remains to check that (IM) holds. 

Let $I$ be independent and a set $X\subseteq V$ such that $I\subseteq X$ be given. Suppose there is a ``red'' linkage from $I$ to $B_0$. Apply the Aharoni-Berger-Menger's theorem on $X$ and $B_0$ to get a ``blue'' linkage $\Qcal$ from $B \subseteq X$ to $B_0$ and an $X$--$B_0$ separator $S$ on the blue linkage. 
Let $H$ be the subgraph induced by those vertices separated from $B_0$ by $S$ with the edges going out of $S$ deleted. 
Since every linkage from $H$ to $B_0$ goes through $S$,  
a subset of $V(H)$ is linkable in $(D, B_0)$ if and only if it is linkable in $(H, S)$. Use the \hyperref[thm:linkageThm]{linkage theorem} to find a linkage $\Qcal^\infty$ from $X^\infty$ with $I \subseteq X^\infty \subseteq I \cup B \subseteq X$ onto~$S$.

Let $Y\supseteq X^\infty$ be any independent set in $M_L(H,S)$. By applying the linkage theorem on a linkage from $Y$ to $S$ and $\Qcal^\infty$ in $(H,S)$, we may assume that $Y$ is linkable onto $S$ by a linkage $\Qcal'$.  
Concatenating $\Qcal'$ with segments of paths in $\Qcal$ starting from $S$ and adding trivial paths from  $B_0\- V(\Qcal)$ gives us a linkage from $Y\cup (B_0\- V(\Qcal))$ onto $B_0$. By the hypothesis, $Y\cup (B_0\- V(\Qcal))$ is a maximal independent set in $M_L(D, B_0)$. 

Applying the above statement on $X^\infty$ shows that $X^\infty\cup (B_0\- V(\Qcal))$ is also maximal in $M_L(D, B_0)$. It follows that $Y=X^\infty$. Hence, $X^\infty$ is maximal in $M_L(H, S)$, and so also in $M_L(D, B_0) \cap 2^X$. 
This completes the proof that $M_L(D, B_0)$ is a matroid. 
\end{proof}

Next we show that containing an AC is the only reason that the criterion \eqref{eqn:dagger} fails. 

\begin{lem}\label{thm:gammoidBaseCriterion}
Let $(D, B_0)$ be a dimaze without any AC. Then a set $B \subseteq V$ is maximal in $M_L(D, B_0)$ if and only if it is linkable onto $B_0$.
\end{lem}

\begin{proof}
The forward direction follows trivially from \eqref{eqn:linkableOnto}. 

For the backward direction, let $I$ be a non-maximal subset that is linkable onto $B_0$, by a ``blue'' linkage $\Qcal$. Since $I$ is not maximal, there is $x_0\notin I$ such that $I + x_0$ is linkable to $B_0$ as well, by a ``red'' linkage $\Pcal$.
Construct an AC inductively as follows:

Use (the proof of) the \hyperref[thm:linkageThm]{linkage theorem} to get a linkage $\Qcal^\infty$ from $I+x_0$ onto $B_0$. 
Since $Y_\Pcal \subseteq Y_\Qcal$ and $X_\Qcal \subseteq X_\Pcal$, $\Acal^\infty=\Ccal^\infty=\emptyset$. So each path in $\Qcal^\infty$ consists of a red initial and a blue terminal segment. 

Start the construction with $x_0$.
For $k \geq 1$, if $x_{k-1}$ is defined, let $Q_k$ be the blue path containing $p_{k-1} := f^\infty_{x_{k-1}}$. We will prove that $p_{k-1}\notin I$ so that we can choose a last vertex $q_k$ on $Q_k\!\stackrel{\circ}{p_{k-1}}$ that is on a path in $\Qcal^\infty$.  Since the blue segments of $\Qcal^\infty$ are disjoint, $q_k$ lies on a red path $P_{x_k}$. We continue the construction with $x_k$.

\begin{clm}
For each $k\geq 1$, $p_{k-1}\notin I$ and hence, the blue segment $q_k Q_k p_{k-1}$ is non-trivial. The red segment $q_kP_{x_k}p_k$ is also non-trivial.
\end{clm}
\begin{proof}
We prove by induction that $p_{k-1}\notin I$. This will guarantee that $q_kQ_k p_{k-1}$ is non-trivial since $Q_k\cap I\in V(\Qcal^\infty)$. Clearly, $p_0\notin I$. Given $k\geq 1$, assume that $p_{k-1}\notin I$. 
We argue that $q_k\neq p_k$. 
Suppose not for a contradiction. Then the path $P_{x_k} q_k Q_k$ is in $\Bcal^\infty$. Since $q_kQ_k p_{k-1}$ is non-trivial,  $p_{k-1}$ and $p_k$ are distinct vertices of the form $f^\infty_x$ on $P_{x_k}q_k Q_k$. This contradicts that $P_{x_k}q_kQ_k$ is in $\Bcal^\infty$. 
Hence, we have $p_k\neq q_k$,\footnote{On day $i_{k-1}$, the farmer at $p_{k-1}$ prevents the water supplied by $Q_k$ from flowing to $q_k$. Hence, the farmer on $P_{x_k}$ must go further to get water, so $p_k\neq q_k$.} and so $p_k\notin I$. This also shows that the red segment $q_k P_{x_k}p_k$ is non-trivial. 
\end{proof}

\begin{clm} 
\label{thm:keyDisjoint}
For any $j< k$, $x_j \neq x_k$. Therefore, $q_k Q_k p_{k-1}$ is disjoint from $q_j Q_j p_{j-1}$, and so is $q_k P_{x_k}p_k$ from $q_j P_{x_j} p_j$. 
\end{clm}
\begin{proof}
For $k\geq 0$, let $i_k$ is the least integer such that $f_{x_k}^{i_k} = f_{x_k}^\infty$. We now show that $i_{k-1} < i_k$ for any given $k\geq 1$. By the choice of $q_k$, for any $i\leq i_{k-1}$, $f^{i}_{x_k}$ is on the segment $P_{x_k}q_k$. Since $P_{x_k} q_k$ is a red segment of $\Qcal^\infty$, $p_k$ is in the segment $q_k P_{x_k}$. 
Since $p_k\neq q_k$, it follows that $P_{x_k}f_{x_k}^{i_{k-1}}\subseteq P_{x_k}q_k \subsetneq P_{x_k} f_{x_k}^{\infty}$. This implies that $f_{x_k}^{i_{k-1}}\neq f_{x_k}^\infty$, so that by  definition of $i_k$, we have $i_k > i_{k-1}$.\footnote{Before day $i_{k-1}$, the water supplied by $Q_k$ ensures that the farmer on $P_{x_k}$ need not go beyond $q_k$. But eventually he moves past $q_k$ and arrives at $p_k$ on day $i_k$; so $i_k>i_{k-1}$.}
Hence, $x_k\neq x_j$ for any $j\neq k$. 
\end{proof}

We now show that $\bigcup_{k=1}^\infty q_kQ_k \cup q_kP_{x_k}p_k$ is an AC. Indeed, by the claims, it remains to check that $q_k P_{x_k} p_k$ does not meet any $q_j Q_j$ for any $j<k$. But if there is such an intersection, it can neither lie in $\stackrel{\circ}{q_j}\!\! Q_j\!\! \stackrel{\!\!\!\!\!\!\!\circ}{p_{j-1}}$ by the choice of $q_j$; 
nor in $p_{j-1}Q_j$ since $x_k\neq x_{j-1}$ (\autoref{thm:keyDisjoint}) and $\Qcal^\infty$ is a linkage. 
Hence, we have constructed an AC. This contradiction shows that $I$ is maximal.
\end{proof}

We have all the ingredients to prove the main result. 

\begin{mainthm}\label{thm:noAC2Gammoid}
Given a dimaze, the vertex sets linkable onto the exits form the bases of a matroid if and only if the dimaze contains no alternating comb. The independent sets of this matroid are precisely the linkable sets of vertices. 
\end{mainthm}
\begin{proof}
The backward direction of the first statement follows from \autoref{thm:dagger2matroid} and \autoref{thm:gammoidBaseCriterion}. To see the forward direction, suppose there is an alternating comb $C$. Let $B_1$ be the union of the vertices of out-degree 2 on $C$ with $B_0-C$. Then $B_1$ is linkable onto $B_0$, and so is $B_1+v$ for any $v\in B_0\cap C$. But $B_1$ and $B_1+v$ violate the base axiom (B2). 
The second statement follows from the first and \eqref{eqn:linkableOnto}. 
\end{proof}

\begin{cor} Any dimaze which does not define a matroid contains an AC. 
\end{cor}

\subsection{Nearly finitary linkability system}
\label{sec:NFLinkability}
Although forbidding AC ensures that we get a strict gammoid, not every strict gammoid arises this way. It turns out that when a dimaze gives rise to a nearly finitary (\cite{ACF}) linkability system, the dimaze defines a matroid regardless of whether it contains an AC or not. We will show this using the proof of the linkage theorem. 

\begin{lem}\label{thm:I4_ForLinkabilitySystems}
Let $(D, B_0)$ be a dimaze. Then $M_L(D, B_0)$ satisfies the following:
\begin{enumerate}
\item[$(\ast)$] For all independent sets $I$ and $J$ with $J\- I\neq \emptyset$, for every $v\in I\- J$ there exists $u \in J \setminus I$ such that $J + v - u$ is independent.
\end{enumerate}
\end{lem}

\begin{proof}
We may assume that $I \setminus J=\{v\}$. Let $\Qcal=(Q_y)_{y\in Y_\Qcal}$ be a ``blue'' linkage from $J$ onto some $Y_\Qcal \subseteq B_0$ and $\Pcal$ a ``red'' one from $I$. 
The linkage theorem yields a linkage 
$\Qcal^\infty$, which we will show to witness the independence of a desired set. 
We use the notations introduced in its proof. 
For each $y\in Y_\Qcal$, let $t^0_y$ be the initial vertex of $Q_y$. 

For $i>0$ it is not hard to derive the following facts from the definitions of $\Qcal^i, f^i_x$ and $t^i_y$: 
\begin{eqnarray}
\label{eqn:happyDontMove}
x \in I\cap \ini(\Qcal^{i-1}) &\Longrightarrow & f^i_x = f^{i-1}_x;\\
\label{eqn:CoveredByAi}
t^0_y\in \ini(\Acal^{i}) &\iff & \forall x\in I, f^{i}_x\notin Q_y;\\
\label{eqn:notCovered}
x\in I\- \ini(\Qcal^i) &\iff & \exists y\in Y_\Qcal, x'\in I \mbox{ s.t. } f^i_x \in Q_y\stackrel{\circ}{f^i_{x'} }.
\end{eqnarray}

\begin{clmnn}
For $i\geq 0$, either $\Qcal^i = \Qcal^\infty$ or there is some $x^i$ such that:\footnote{There is a unique unhappy farmer and the others are distributed such that there is exactly one happy farmer on each blue pipeline leading to a farm.} 
\begin{itemize}
\item[$U_i$:] 
$(J + v) \- \ini(\Qcal^i) = \{ x^i \}$ and $I- x^i \subseteq \ini( \Bcal^i);$ 
\item[$D_i$:]
$\forall y\in Y_\Qcal$, if $t^0_y \in I$ then $\exists! x \in I - x^i$ s.t.  $f^i_x \in Q_y$; no such $x$ otherwise.  
\end{itemize}
\end{clmnn}
\begin{proof}
With $x^0:=v$, the claim clearly holds for $i=0$.  
Given $i>0$, to prove the claim, we may assume that 
$\Qcal^{i-1} \neq \Qcal^\infty$ and $U_{i-1}$ and $D_{i-1}$ hold. 

By definition of $f^i_{x^{i-1}}$, either 
\begin{align*}
 f^i_{x^{i-1}} \in t^{i-1}_{y^i} Q_{y^i} \text{ for some unique } y^i \in Y_\Qcal \text{ or } f^i_{x^{i-1}} \in Y_\Pcal\- Y_\Qcal.
\end{align*}
Note that by \eqref{eqn:happyDontMove} only $x^{i-1}$ can be a vertex such that $f^i_{x^{i-1}} \neq f^{i-1}_{x^{i-1}}$. Hence, $t^i_y = t^{i-1}_y$ for all $y \in Y_\Qcal$ except possibly $y^i$ which satisfies $t^i_{y^i} = f^i_{x^{i-1}}$. So by \eqref{eqn:notCovered}, we have $x^{i-1} \in \ini(\Qcal^i)$.\footnote{The unhappy farmer becomes happy, either by repairing or because he is the only one who moves to steal water.}

\medskip 
Case (i): Suppose that there exists $x\in I - x^{i-1}$ such that $f^i_{x^{i-1}}$ and $f^i_{x}$ are on the same path $Q_{y^i}$. 
By $D_{i-1}$, $t^0_{y^i} \in I$ and $x$ is unique. Then $D_{i}$ holds for $x^i:=x$. In particular, by \eqref{eqn:CoveredByAi}, $J\- I \subseteq \ini(\Acal^i)$.\footnote{The whole scene only changes at $Q_{y^i}$, so there is still exactly one unhappy farmer and the pipelines not leading to a farm are untouched.}  

We now prove $U_i$. As $x^i\in \ini(\Bcal^{i-1})$, $t^{i-1}_{y^i} = f^{i-1}_{x^i}$, so $f^i_{x^{i-1}} \in \stackrel{\circ}{f^i_{x^i}}\!\! Q_{y^i}$, which implies that $x^i \notin \ini(\Qcal^i)$ 
by \eqref{eqn:notCovered}.\footnote{The farmer at $f^{i}_{x^i}$ becomes unhappy as his water has been stolen by the farmer at $f^{i}_{x^{i-1}}$.}
Given $x\in I-\{x^{i-1}, x^i\}$, then $x\in \ini(\Bcal^{i-1})$ by $U_{i-1}$. So there exists $y\neq y^i$, such that 
$f^i_x = f^{i-1}_x = t^{i-1}_y = t^i_y$. It follows that $x\in \ini(\Bcal^i)$, and $I- x^i \subseteq \ini(\Bcal^i)$. Therefore, $(J+v)\- \ini(\Qcal^i) = \{x^i\}$. 

Case (ii): Suppose that there does not exist any $x\in I - x^{i-1}$ such that $f^i_x$ is on the path $Q_{y^i}$ containing $f^i_{x^{i-1}}$, if such a path exists.     
In this case, $t^0_{y^i}\in J\- I$. 
By $D_{i-1}$, \eqref{eqn:happyDontMove} and \eqref{eqn:notCovered}, we have $I - x^{i-1}\subseteq \ini(\Bcal^i)$. 
Hence, $I \subseteq \ini(\Qcal^i)$, and $\Qcal^\infty= \Qcal^i$. 
\end{proof}

If for some integer $i>0$, case (ii) holds, then by \eqref{eqn:CoveredByAi}, only $u := t^0_{y^i} \in J \- I$ can fail to be in~$\ini(\Acal^\infty)$. 
Otherwise, case (i) holds for each integer $i\geq 0$, so that $J\- I$ is a subset of $\Ini(\Acal^i)$ and hence a subset of $\ini(\Qcal^\infty)$. In either situation, since $I = X_\Pcal \subseteq \ini(\Qcal^\infty)$, 
we conclude that there is some $u\in J\- I$ such that $J + v-u$ is independent. 
\end{proof}

Let $(E, \Ical)$ be a set system. Recall that it is {\em finitary} if a set is in $\Ical$ as soon as all its finite subsets are. The {\em finitarisation} $(E,\Ical)^{\mbox{\small fin}}$ of $(E, \Ical)$ consists of the sets which have all their finite subsets in $\Ical$. $(E,\Ical)$ is called {\em nearly finitary} if for any maximal element $B\in(E, \Ical)^{\mbox{\small fin}}$ there is an $I\in\Ical$ such that $|B\setminus I|<\infty$.

\begin{thm}\label{thm:NearlyFinitaryGammoid}
Let $(D,B_0)$ be a dimaze. If $M_L(D,B_0)$ is nearly finitary, then it is a matroid. 
\end{thm}
\begin{proof}
Since $M_L(D, B_0)$ satisfies (I1), (I2) and $(\ast)$, by \cite[Lemma 4.15]{ACF}, it also satisfies (IM). Hence, by \autoref{thm:I3_ForLinkabilitySystems}, it is a matroid.  
\end{proof}

The theorem shows that dimazes which contain an AC may also define matroids. An \emph{comb} is a dimaze consisting of a directed (double) ray $R$ and infinitely many non-trivial disjoint directed paths from $R$, which meet $R$ only at the initial vertices, onto the set of exits $B_0$. 

\begin{exm}
We construct a dimaze $(D, B_0)$ which defines a nearly finitary linkability system, by identifying the corresponding exits of $n>1$ copies of a comb defined via a ray without terminal vertex (see \autoref{fig:turbine}).   
Note that $(D, B_0)$ contains an AC; and $M_L(D, B_0)$ is not finitary (a vertex not in $B_0$ together with all reachable vertices in $B_0$ form an infinite circuit\footnote{Any 
circuit in a finitary matroid is finite.}). 

We check that $M_L(D, B_0)$ is nearly finitary. 
Let $B$ be a maximal element in $M_L(D, B_0)^{\mbox{\small fin}}$. Let $I$ be the set obtained from $B$ by deleting the last vertex, if exists, of $B$ on each ray in $D- B_0$; and $T := B\- I$. 
Fix an enumeration for $I$. Note that for any integer $k\geq 1$, there are only finitely many linkages from $[k]$ to $B_0$ disjoint from $T$. In fact, there is at least one: the restriction to $[k]$ of a linkage of the finite subset $[k]\cup T$ of $B$. Applying the infinity lemma (\cite[Proposition 8.2.1]{Die10}), with the $k$th set consisting of the finite non-empty collection of linkages from $[k]$ to $B_0$ disjoint from $T$, we obtain a linkage from $I$ to $B_0$. Hence, $I\in M_L(D, B_0)$. As $B$ is arbitrary and $|T|\leq n$, we conclude that $M_L(D, B_0)$ is nearly finitary. 
\end{exm}

\begin{figure}
 \centering
 \includegraphics[width = 10cm]{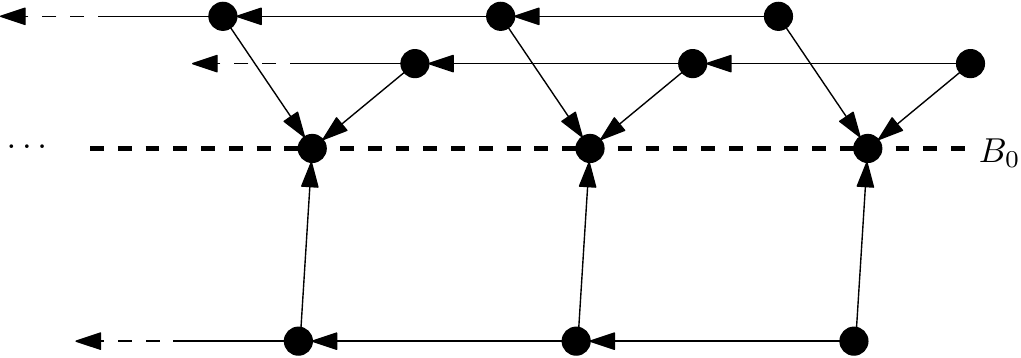}
 \caption{A dimaze that defines a nearly finitary linkability system}
 \label{fig:turbine}
 \end{figure} 

On the other hand, \autoref{thm:NearlyFinitaryGammoid} does not imply \autoref{thm:noAC2Gammoid}. 

\begin{exm}
\label{thm:graph101}
Take infinitely many disjoint copies of a comb on a directed double ray. Add two extra vertices, and an edge from every vertex on the double ray of each comb to those two vertices. Take all the vertices of out-degree 0 as the exits. Then this dimaze defines a matroid that is 3--connected, not nearly finitary and whose dual is not nearly finitary. 
\end{exm}

So far we have seen that if a dimaze $(D, B_0)$ does not contain any AC or that $M_L(D, B_0)$ is nearly finitary, then $M_L(D, B_0)$ is a matroid. However, there are examples of strict gammoids that lie in neither of the two classes. All our examples of dimazes that do not define a matroid share another feature other than possessing an AC: there is an independent set $I$ that cannot be extended to a maximal in $I\cup B_0$. In view of this, we propose the following. 

\begin{con}
Suppose that for all $I\in M_L(D, B_0)$ and $B\subseteq B_0$, there is a maximal independent set in $I\cup B$ extending $I$. Then (IM) holds for~$M_L(D, B_0)$. 
\end{con}

\section{Dimazes with alternating combs}
\label{sec:withAR}
We have seen in \autoref{sec:strictGammoids} that forbidding AC in a dimaze guarantees that it defines a strict gammoid. However, the AC in \autoref{fig:ARIC} defines a finitary strict gammoid. On the other hand, this strict gammoid is isomorphic to the one defined by an \emph{incoming comb} via the isomorphism given in the figure. So one might think that every strict gammoid has a defining dimaze which does not contain any AC. We will prove that this is not the case with two intermediate steps. In \autoref{sec:connectivity}, we derive a  property satisfied by any strict gammoid defined by a dimaze without any AC. In \autoref{sec:treesTransversal}, we show that any tree defines a transversal matroid. We then construct a strict gammoid which cannot be defined by a dimaze without any AC. 

\begin{figure}[hbt]
 \centering
 \includegraphics{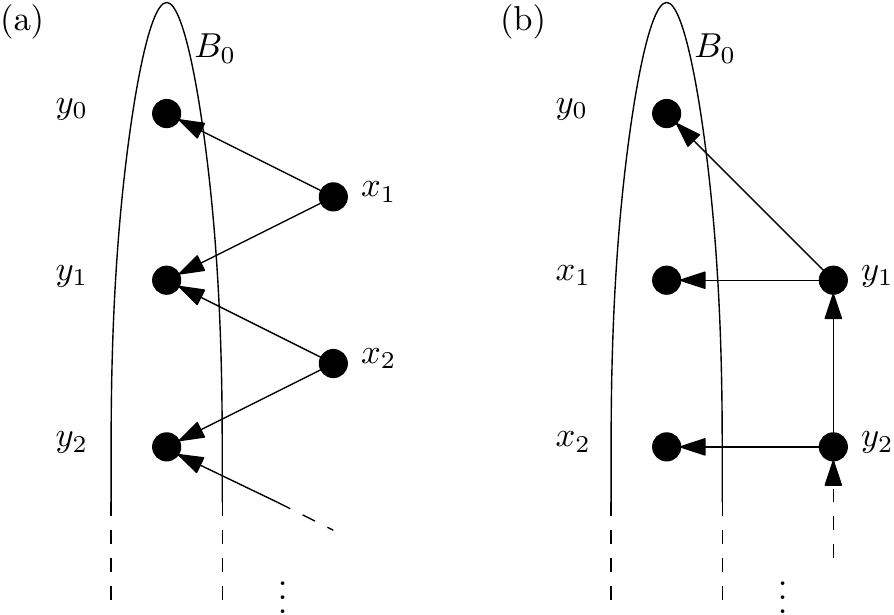}
 \caption{An alternating comb and an incoming comb defining isomorphic strict gammoids}
 \label{fig:ARIC}
 \end{figure}

\subsection{Connectivity}
\label{sec:connectivity}
Connectivity in finite matroids stems from graph connectivity and is a well established part of the theory. In the infinite setting, Bruhn and Wollan \cite{BW12} gave the following rank-free definition of connectivity that is compatible with the finite case. For an integer $k\geq 0$, a \emph{$k$--separation} of a matroid is a partition of $E$ into $X$ and $Y$ such that both $|X|, |Y|\geq k$ and for any pair of bases $B_X, B_Y$ of $M\- Y$ and $M\- X$ respectively, the number of elements to be deleted from $B_X\cup B_Y$ to get a base of $M$ is less than $k$. It was shown there that this number does not depend on the choice of $B_X$ or $B_Y$ or the deleted set. A matroid is \emph{$k$--connected} if there are no $l$--separations for any $l<k$. 
If a matroid does not have any $k$--separations for any integer $k$, then it is \emph{infinitely connected}. 
Recall that the only infinitely connected finite matroids are uniform matroids of rank about half of the size of the ground set (see \cite[Chapter 8]{Oxl92}) and they are strict gammoids. It seems natural to look for an infinitely connected infinite matroid among strict gammoids, but here we give a partial negative result. It remains open whether there is an infinitely connected infinite gammoids. 

\begin{lem}
\label{thm:noFiniteCC}
If a dimaze $(D, B_0)$ does not contain any AC, then $M_L(D, B_0)$ contains a finite circuit or cocircuit. 
\end{lem}
\begin{proof}
Suppose the lemma does not hold. Then every finite subset of $V$ is independent and coindependent, and $B_0$ is infinite. We construct a sequence $(R_k: k\geq 1)$ of subdigraphs of $D$ that gives rise to an AC for a contradiction. 

Let $v_1\notin B_0$ and $R_1:=v_1$. For $k\geq 1$, we claim that there is a path $P_k$ from $v_k$ to $B_0$ such that $P_k\cap V(R_k) = \{v_k\}$, a vertex $w_k$ on $\stackrel{\circ}{v_k}\!\!P_k$, and a vertex $v_{k+1}\notin V(R_k)\cup P_k$ with $(v_{k+1}, w_k)\in E(D)$. Let $R_{k+1}:=R_k\cup P_k \cup (v_{k+1}, w_k)$.  

Indeed, since any finite set containing $v_k$ is independent, there is a path from $v_k$ avoiding any given finite set disjoint from $v_k$. Hence, there is a set $\Fcal$ of $|V(R_k)|+1$ disjoint paths (except at $v_k$) from $v_k$ to $B_0$ avoiding the finite set $V(R_k) -v_k$. 
Since $V(\Fcal)\cup R_k$ is coindependent, its complement contains a base $B$, witnessed by a linkage $\Pcal$. 
Since $|V(\Fcal)\cap B_0|> |V(R_k)|$ and $\Ter(\Pcal)= B_0$, there is a path $P\in \Pcal$ that is disjoint from $R_k$ and ends in $V(\Fcal)\cap B_0$. Then the last vertex $v_{k+1}$ of $P$ before hitting $V(\Fcal)$, the next vertex $w_k$, and the segment $P_k:= w_kP$ satisfy the requirements of the claim. By induction, the claim holds for all $k\geq 1$. 

Let $R:=\bigcup_{k\geq 1} R_k$. Then $(R, V(R)\cap B_0)$ is an AC in $(D, B_0)$. This contradiction completes the proof.
\end{proof}

In an infinite matroid that is infinitely connected, the bipartition of the ground set into any finite circuit of size $k$ against the rest is a $k$--separation. Hence, such a matroid must not have finite circuits or cocircuits. 
\begin{cor}
\label{thm:noARconnectivity}
If an infinite dimaze $(D, B_0)$ does not contain any AC, then $M_L(D, B_0)$ is not infinitely connected.
\end{cor}

\subsection{Trees and transversal matroids}
\label{sec:treesTransversal}
In this section, the aim is to prove that a tree defines a transversal matroid. We then construct a strict gammoid which cannot be defined by a dimaze without any AC. 

Given a bipartite graph $G$, fix an ordered bipartition $(V,W)$ of $V(G)$; this induces an ordered bipartition of any subgraph of $G$. A subset of $V$ is \emph{independent} if it is matchable to $W$. Let $M_T(G)$ be the pair of $V$ and the collection of independent sets. It is clear that (I1), (I2) hold for $M_T(G)$. When $G$ is finite, (I3) also holds \cite{EF65}. The proof of this fact which uses alternating paths can be extended to show that (I3) also holds when $G$ is infinite. 

Let $m$ be a matching. An edge in $m$ is called an $m$--edge. An $m$--\emph{alternating path} is a path or a ray that starts from a vertex in $V$ such that the edges alternate between the $m$--edges and the non-$m$--edges. An $m$--$m'$ \emph{alternating path} is defined analogously with $m'$, also a matching, replacing the role of the non-$m$--edges. 

\begin{lem}
For any bipartite graph $G$, $M_T(G)$ satisfies (I3).
\end{lem}
\begin{proof}
Let $I, B\in M_T(G)$ such that $B$ is maximal but $I$ is not. 
As $I$ is not maximal, there is a matching $m$ of $I+x$ for some $x\in V\- I$. Let $m'$ be a matching of $B$ to $W$.  Start an  $m$--$m'$ alternating path $P$ from $x$. 
By maximality of $B$, the alternating path is not infinite and cannot end in $W\- V(m')$. So we can always extend it until it ends at some $y\in B\- I$. Then $m\Delta E(P)$ is a matching of $I+y$, which completes the proof. 
\end{proof}

If $M_T(G)$ is a matroid, it is called a \emph{transversal matroid}. 
For $X\subseteq V$, the restriction of $M_T(G)$ to $X$ is also a transversal matroid, and can be defined by the independent sets of the subgraph of $G$ induced by $X\cup N(X)$.

Suppose now $G$ is a tree rooted at a vertex in $W$. By upwards (downwards), we mean towards (away from) the root. For any vertex set $Y$, let $N^{\uparrow}(Y)$ be the upward neighbourhood of $Y$, and $N^{\downarrow}(Y)$ the set of downward neighbours. An edge is called \emph{upward} if it has the form $\{v, N^\uparrow(v)\}$ where $v\in V$, otherwise it is \emph{downward}. 

We will prove that $M_T(G)$ is a matroid.
For a witness of (IM), we inductively construct a sequence of matchings $(m^\alpha :\alpha \geq 0)$, indexed by ordinals, of $I^\alpha := V(m^\alpha) \cap V$.

Given $m^{\beta-1}$, to define a matching for $\beta$, we consider the vertices in $V\- I^{\beta-1}$ that do not have unmatched children for the first time at step $\beta-1$. We ensure that any such vertex $v$ that is also in $I$ is matched in step $\beta$, by exchanging $v$ with a currently matched vertex $r_v$ that is not in $I$. 

When every vertex that has not been considered has an unmatched child, we stop the algorithm, at some step $\gamma$. We then prove that the union of all these unconsidered vertices and $I^\gamma$ is a maximal independent superset of $I$. 

\begin{thm}
\label{thm:transversalTree}
For any tree $G$ with an ordered bipartition $(V,W)$, $M_T(G)$ is a transversal matroid. 
\end{thm} 
\begin{proof}
To prove that $M_T(G)$ is a matroid, it suffices to prove that (IM) holds. Let an independent set $I\subseteq X\subseteq V$ be given. Without loss of generality, we may assume that $X=V$. 

We start by introducing some notations. 
Root $G$ at some vertex in $W$. 
Given an ordinal $\alpha$ 
and a matching $m^\alpha$, let
$I^\alpha := V(m^\alpha) \cap V$ and $W^\alpha := V(m^\alpha) \cap W$. 
Given a sequence of matchings $(m^{\alpha'}: \alpha'\leq \alpha)$,
let \begin{align*}
\label{eqn:Cbeta}
 C^{\alpha} := \{ v\in V\- I^\alpha: N^\downarrow(v)\subseteq W^\alpha \mbox{ but } N^\downarrow(v)\not\subseteq W^{\alpha'}\  \forall \alpha' < \alpha\}.
 \end{align*}
Note that $C^\alpha \cap C^{\alpha'} = \emptyset$ for $\alpha'\neq \alpha$. 
For each $w\in W\-W^\alpha$, choose one vertex $v_w$ in $N^\downarrow(w)\cap C^\alpha$ if it is not empty. Let 
\[S^\alpha := \{ v_w: w\in W\- W^\alpha \mbox{ and } N^\downarrow(w)\cap C^\alpha\neq \emptyset\}.\]

Denote the following statement by $A(\alpha)$: 
\begin{quote}
There is a pairwise disjoint collection $\Pcal^\alpha:=\{ P_v: v\in I\cap C^\alpha\- S^\alpha\}$ of $m^\alpha$--alternating paths such that each $P_v$ starts from $v\in I\cap C^\alpha\- S^\alpha$ with a downward edge and ends at the first vertex $r_v$ in $I^\alpha\- I$. 
\end{quote}

Start the inductive construction with $m^0$, which is the set of upward edges that is contained in every matching of $I$. 
It is not hard to see that $C^0\cap I=\emptyset$, so that $A(0)$ holds trivially. 

Let $\beta>0$. 
Given the constructed sequence of matchings $(m^\alpha: \alpha < \beta)$, suppose that $A(\alpha)$ holds for each $\alpha< \beta$. 
Construct a matching $m^\beta$ as follows.

If $\beta$ is a successor ordinal, 
let 
\begin{align*}
 m^{\beta} := E(S^{\beta-1}, N^\uparrow(S^{\beta-1})) \cup (m^{\beta-1} \Delta 
E(\Pcal^{\beta-1})).
\end{align*}
By $A(\beta-1)$, the paths in $\Pcal^{\beta-1}$ are disjoint. So $m^{\beta-1} \Delta E(\Pcal^{\beta-1})$ is a matching. 
Using the definition of $S^{\beta-1}$, we see that $m^{\beta}$ is indeed a matching. 
Observe also that  
\begin{eqnarray}
 \label{eqn:Iincreasing0} 
 I^{\beta-1}\cap I & \subseteq & I^{\beta}\cap I;\\
 \label{eqn:Wincreasing0}
 W^{\beta-1} & \subseteq & W^{\beta-1} \cup N^\uparrow(S^{\beta-1})=W^{\beta}.
\end{eqnarray}

If $\beta$ is a limit ordinal, define $m^\beta$ by
\begin{align}
\label{eqn:mLimit}
e\in m^\beta \iff \exists \beta' < \beta \mbox{ such that } e\in m^\alpha\  \forall \alpha \mbox{ with } \beta'\leq \alpha< \beta.
\end{align}
As $m^\alpha$ is a matching for every ordinal $\alpha < \beta$, we see that $m^\beta$ is a matching in this case, too. 

\medskip
Suppose that a vertex $u\in (V\cap I)\cup W$ is matched to different vertices by $m^\alpha$ and $m^{\alpha'}$ for some $\alpha, \alpha'\leq  \beta$. 
Then there exists some ordinal $\alpha''+1$ between $\alpha$ and $\alpha'$ such that $u$ is matched by an upward $m^{\alpha''}$--edge and by a downward $m^{\alpha''+1}$--edge. Hence, the change of the matching edges is unique. This implies that for any $\alpha, \alpha'$ with $\alpha\leq \alpha' \leq \beta$, by \eqref{eqn:Iincreasing0} and \eqref{eqn:Wincreasing0}, we have
\begin{eqnarray}
\label{eqn:Iincreasing}
I^\alpha \cap I & \subseteq & I^{\alpha'}\cap I;\\
\label{eqn:Wincreasing}
W^\alpha & \subseteq & W^{\alpha'}.
\end{eqnarray}
Moreover, for an upward $m^\beta$--edge $vw$ with $v\in V$, we have 
\begin{align}
\label{eqn:UpwardEdgeStays}
v \in I^0 \mbox{ or } \exists \alpha < \beta \mbox{ such that } v\in C^\alpha \mbox{ and } w\notin W^{\alpha}.
\end{align}

We now prove that $A(\beta)$ holds. 
Given $v_0=v\in I\cap C^\beta \- S^\beta$, 
we construct a decreasing sequence of ordinals starting from $\beta_0 := \beta$.  
For an integer $k\geq 0$, suppose that $v_k \in I\cap C^{\beta_k}$ with $\beta_k\leq \beta$ is given. 
By \eqref{eqn:Iincreasing}, $I^0\subseteq I^{\beta_k}$, so $v_k\notin I^0$ and hence  
there exists $w_k \in N^\downarrow(v_k)\- W^0$.\footnote{For a vertex $v\notin I$, $N^\downarrow(v)\setminus W^0$ may be empty.} Since $N^{\downarrow}(v_k)\subseteq W^{\beta_k}\subseteq W^\beta$, $w_k$ is matched by $m^\beta$ to some vertex $v_{k+1}$. In fact, 
as $w_k\notin W^0$, $v_{k+1}\notin I^0$. 
Let $\beta_{k+1}$ be the ordinal with $v_{k+1} \in C^{\beta_{k+1}}$. Since $v_{k+1}w_k$ is an upward edge and $N^\downarrow(v_k)\subseteq W^{\beta_k}$, we have by \eqref{eqn:UpwardEdgeStays} that $w_k\in W^{\beta_k}\- W^{\beta_{k+1}}$. By \eqref{eqn:Wincreasing}, $\beta_k >\beta_{k+1}$.

As there is no infinite decreasing sequence of ordinals, we have an $m^\beta$--alternating path $P_v=v_0 w_0 v_1 w_1 \cdots$ that stops at the first vertex $r_v \in V \- I$.

The disjointness of the $P_v$'s follows from that every vertex has a unique upward neighbour and, as we just saw, that $\stackrel{\circ}{v}\!\! P_v$ cannot contain any vertex $v' \in C^\beta$. So $A(\beta)$ holds. 

We can now go onwards with the construction. 

\medskip
Let $\gamma \leq |V|$\footnote{For example, fix a well ordering of $V$ and map each $\beta$ to the least element in $C^\beta$.} be the least ordinal such that $C^\gamma = \emptyset$. Let $C:=\bigcup_{\alpha<\gamma} C^\beta$ and $U:=V\- (I^0 \cup C)$; so $V$ is partitioned into $I^0, C$ and $U$. As $C^\gamma =\emptyset$, every vertex in $U$ can be matched downwards to a vertex that is not in $W^\gamma$. These edges together with $m^\gamma$ form a matching $m^B$ of $B:=U\cup I^\gamma$, which we claim to be a witness for (IM). 
By \eqref{eqn:Iincreasing}, $I^0\cup (C\cap I)\subseteq I^\gamma$, hence, $I\subseteq B$.

Suppose $B$ is not maximally independent for a contradiction. Then there is an $m^B$--alternating path $P=v_0w_0v_1w_1\cdots$ such that $v_0\in V\- B$ that is either infinite or ends with some $w_n\in W\- V(m^B)$. We show that neither occurs.

\begin{clm}
\label{thm:Pfinite}
$P$ is finite. 
\end{clm}
\begin{proof}
Suppose $P$ is infinite. Since $v_0\notin B$, $P$ has a subray $R=w_iP$ such that $w_iv_{i+1}$ is an upward $m^B$--edge. So  $w_j v_{j+1}\in m^B$ for any $j\geq i$. As vertices in $U$ are matched downwards, $R\cap U=\emptyset$. As $m^B\Delta E(R)$ is a matching of $B\supseteq I$ in which every vertex in $R\cap V$ is matched downwards, $R\cap I^0 = \emptyset$ too. 
So for any $j\geq i$, there exists a unique $\beta_j$ such that $v_j\in C^{\beta_j}$. 

Choose $k\geq i$ such that $\beta_k$ is minimal. But with a similar argument used to prove $A(\beta)$, we have $\beta_k > \beta_{k+1}$. Hence $P$ cannot be infinite.
\end{proof}

\begin{clm}
\label{thm:Pends}
$P$ does not end in $W\- V(m^B)$. 
\end{clm}
\begin{proof}
Suppose that $P$ ends with $w_n\in W\-V(m^B)$. 
Certainly, $v_n$ can be matched downwards (either to $w_{n-1}$ or $w_n$) in a matching of $B\supseteq I$. Hence, $v_n\notin I^0$. 
It is easy to check that for $v\in C^\alpha$, $N(v)\subseteq W^{\alpha+1}$. Hence, as $w_n\in W\- W^\gamma$, $v_n\notin C$. 
Hence, $v_n\in U$. It follows that for each $0<i\leq n$, $v_i$ is matched downwards and so does not lie in $I^0$. 
As $v_0\notin B$, $v_0\in C$. It follows that $w_0\in W^\gamma$ and $v_1\in C$. Repeating the argument, we see that $v_n\in C$, which is a contradiction. 
\end{proof}
We conclude that $B$ is maximal. So (IM) holds and $M_T(G)$ is a matroid. 
\end{proof}

\begin{cor}
\label{thm:dimazeTree}
Let $(D, B_0)$ be a dimaze such that the underlying graph of $D$ is a tree and $B_0$ is a vertex class of a bipartition of $D$ with edges directed towards $B_0$. Then $M_L(D, B_0)$ is a matroid. 
\end{cor}
\begin{proof}
By the theorem, we need only present $M_L(D, B_0)$ as a transversal matroid defined on a tree. Define a tree $G$ with bipartition $((V\-B_0)\cup B_0', B_0)$, where $B_0'$ is a copy of $B_0$, from $D$ by ignoring the directions and joining each vertex in $B_0$ to its copy with an edge. It can be easily checked that $M_L(D, B_0)\cong M_T(G)$.
\end{proof}

Consider the countably infinite branching rooted tree, i.e.~a rooted tree such that each vertex has countably many children. Let $B_0$ consist of the root and vertices on every other level. Define $\Tcal$ by directing all edges towards $B_0$. 
\autoref{thm:dimazeTree} shows that $M_L(\Tcal, B_0)$ is a matroid. Clearly, this matroid does not contain any finite circuit. Moreover, as any finite set $C^*$ misses a base obtained by adding finitely many vertices to $B_0\- C^*$, any cocircuit must be infinite. 
We remark that this matroid is self-dual. With \autoref{thm:noFiniteCC}, we conclude the following.

\begin{cor}
Every dimaze that defines a strict gammoid isomorphic to $M_L(\Tcal, B_0)$ contains an AC.
\end{cor}

\end{document}